\documentclass[11pt]{amsart}
\usepackage{geometry}                % See geometry.pdf to learn the layout options. There are lots.
\usepackage{graphicx}
\usepackage{amssymb}
\usepackage{epstopdf}
\usepackage{color}
\newtheorem{theorem}{Theorem}[section]
\newtheorem{lemma}[theorem]{Lemma}
\newtheorem{corollary}[theorem]{Corollary}

\theoremstyle{definition}
\newtheorem{definition}[theorem]{Definition}

\newtheorem{proposition}[theorem]{Proposition}
\newtheorem{remark}[theorem]{Remark}

\numberwithin{equation}{section}

\def\R{\mathbb{R}}

\newcommand{\Z}{\mathbb{Z}}

\newcommand{\p}{\partial}

\newcommand{\CF}{\mathcal{F}}

\newcommand{\BFR}{\mathbf{R}}
\newcommand{\BFN}{\mathbf{N}}

\newcommand{\BBN}{\mathbb{N}}

\newcommand{\ov}{\overline}
\def\mapright#1{\smash{\mathop{\longrightarrow}\limits^{#1}}}

\newcommand{\be}{\begin{equation} }
\newcommand{\ee}{\end{equation}}
\newcommand{\bse}{\begin{subequations}}
\newcommand{\ese}{\end{subequations}}
\def\bea{\begin{eqnarray}}
\def\eea{\end{eqnarray}}

\title[Non-contractible closed geodesics]{Multiplicity of non-contractible closed geodesics on Finsler compact space forms}

\author{Hui Liu${}^\#$}
\address{School of Mathematics and Statistics, Wuhan University, \\ Wuhan 430072, Hubei, China}
\email{huiliu00031514@whu.edu.cn}
\thanks{${}^\#$ Partially supported by NSFC (Nos. 12022111, 11771341) and
the Fundamental Research Funds for the Central Universities (No. 2042021kf1059)}
	
\author{Yuchen Wang${}^\dagger$}
\address{School of Mathematics and Statistics, Central China Normal University,\\ Wuhan 430079, Hubei, China}
\email{wangyuchen@mail.nankai.edu.cn}
\thanks{${}^\dagger$ Partially supported by the funding of innovating activities in Science and Technology of Hubei Province. }

\date{}

\begin{document}

\begin{abstract}
{\it
Let $M=S^n/ \Gamma$ and $h$ be a nontrivial element of finite order $p$ in $\pi_1(M)$, where the integer $n, p\geq2$,
$\Gamma$ is a finite abelian group which acts freely and isometrically on the $n$-sphere and therefore $M$ is diffeomorphic
to a compact space form. In this paper, we prove that for every irreversible Finsler compact space form $(M,F)$ with
reversibility $\lambda$ and flag curvature $K$ satisfying
\[
\frac{4p^2}{(p+1)^2} \big(\frac{\lambda}{\lambda+1} \big)^2 < K \leq 1,\;\;\lambda< \frac{p+1}{p-1},
\]
there exist at least $n-1$ non-contractible closed geodesics of class $[h]$. In addition, if the metric $F$ is bumpy  and
\[
(\frac{4p}{2p+1})^2 (\frac{\lambda}{\lambda+1})^2 < K \leq 1,\;\;\lambda<\frac{2p+1}{2p-1},
\]
then there exist at least $2[\frac{n+1}{2}]$ non-contractible closed geodesics of class $[h]$, which is the optimal lower
bound due to Katok's example. For $C^4$-generic Finsler metrics, there are infinitely many non-contractible closed geodesics of
class $[h]$ on $(M, F)$ if $\frac{\lambda^2}{(\lambda+1)^2} < K \leq 1$ with $n$ being odd, or $\frac{\lambda^2}{(\lambda+1)^2}\frac{4}{(n-1)^2} < K \leq 1$ with $n$ being even. }


\renewcommand{\theequation}{\thesection.\arabic{equation}}
\renewcommand{\thefigure}{\thesection.\arabic{figure}}

\end{abstract}

\maketitle

\section{Introduction} \label{S:Intro}

Denote by $\Gamma$ a finite abelian group acting on the spheres $S^n$ freely and isometrically and $h \in \Gamma$ a non-trivial element of the order $p \geq 2$. The problem we address in this paper is the the existence and multiplicity of non-contractible closed geodesics which are homotopic to $h$ on compact space form $S^{n}/\Gamma$ with irreversible Finsler metrics. It is a typically non-simply connected manifold with the fundamental group $\pi_1(S^{n}/\Gamma)=\Gamma$.

Recall that a Finsler metric on a finite dimensional manifold $M$ is defined by a function $F:TM \to [0, +\infty)$ holding with the following properties:
\begin{enumerate}
\item $F$ is smooth on $TM \setminus \{0\}$,
\item $F(x,\lambda v) = \lambda F(x,v)$ for $\lambda >0, (x,v) \in TM$,
\item the Hessian $\p_{vv} F^2$ is positively definite for any $(x,v) \in TM$ outside the zero section.
\end{enumerate}
In this case, $(M,F)$ is called a {\it Finsler manifold}. $F$ is
{\it reversible} if $F(x,-v)=F(x,v)$ holds for all $v\in T_xM$ and
$x\in M$. $F$ is {\it Riemannian} if $F(x,v)^2=\frac{1}{2}G(x)v\cdot
v$ for some symmetric positive definite matrix function $G(x)\in
GL(T_xM)$ depending on $x\in M$ smoothly.

The length of a curve $\gamma:[a,b] \to M$ on Finsler manifold $(M,F)$ is given by
\be \label{E:length}
L(\gamma) = \int_a^b F(\gamma,\dot{\gamma}) dt.
\ee
In contrast to Riemannian manifolds, the length of a curve may not coincide with its inverse on Finsler manifolds.
Rademacher \cite{Rad04} defined the reversibility of Finsler metrics by
\be
\lambda:=\max\{ F(-X) \mid X \in TM, \; F(X)=1\} \geq 1,\nonumber
\ee
and a Finsler metric is reversible if and only if the equality holds.

Let $\Lambda M$ be the free loop space on $M$ defined by
\begin{equation*}
\label{LambdaM}
  \Lambda M=\left\{\gamma: S^{1}\to M\mid \gamma\ {\rm is\ absolutely\ continuous\ and}\
                        \int_{0}^{1}F(\gamma,\dot{\gamma})^{2}dt<+\infty\right\},
\end{equation*}
endowed with a natural structure of Riemannian Hilbert manifold on which the group $S^1=\R/\Z$ acts continuously by
isometries (cf. Shen \cite{Shen01}).
A closed curve on a Finsler manifold is a closed geodesic if it is locally the shortest path connecting any two nearby points on this
curve. It is well known (cf. Chapter 1 of Klingenberg \cite{Kli78}) that $c$ is a closed geodesic or a constant curve
on $(M,F)$ if and only if $c$ is a critical point of the energy functional \begin{equation*}
\label{energy}
E(\gamma)=\frac{1}{2}\int_{0}^{1}F(\gamma,\dot{\gamma})^{2}dt.
\end{equation*}
As usual, on any Finsler manifold $(M, F)$, a closed geodesic $c:S^1=\mathbb{R}/\mathbb{Z}\to M$ is {\it prime}
if it is not a multiple covering (i.e., iteration) of any other
closed geodesics. Here the $m$-th iteration $c^m$ of $c$ is defined
by $c^m(t)=c(mt)$. The inverse curve $c^{-1}$ of $c$ is defined by
$c^{-1}(t)=c(1-t)$ for $t \in \mathbb{R}$.  Note that the inverse curve $c^{-1}$ of a closed geodesic $c$
need not be a geodesic except $\lambda =1$.
We call two prime closed geodesics
$c$ and $d$ {\it distinct} if there is no $\theta \in (0,1)$ such that
$c(t)=d(t+\theta )$ for all $t\in \mathbb{R}$.
We shall omit the word {\it distinct} when we talk about more than one prime closed geodesic.
For a closed geodesic $c$ on $(M,\,F)$, denote by $P_c$
the linearized Poincar\'{e} map of $c$. Recall that a Finsler metric $F$ is {\it bumpy} if all the closed geodesics
on $(M, \,F)$ are non-degenerate, i.e., $1\notin \sigma(P_c)$ for any closed
geodesic $c$.

The existence and multiplicity of closed geodesics on manifolds is a long-standing issue in the global differential geometry and calculation of
variations. The first mathematical result is due to Birkhoff \cite{Bir27}, in which he proved that there exists at least one closed geodesic on
closed surfaces with arbitrary Riemannian metric. The existence of closed geodesics on arbitrary Riemannian manifold
is due to Lyusternik and Fet \cite{LF51}, where they considered the energy functional on the free loop space and proved that the energy functional
must have critical points with positive energy. Their arguments also carries on the Finsler metrics.

The multiplicity of closed geodesics is much more involved in the topological structure of the free loop spaces.
In a seminal work \cite{GM69}, Gromoll and Meyer obtained infinitely many geometrically distinct closed geodesics
on the Riemannian manifold $M$, provided the Betti number sequence $\{b_p(\Lambda M; \mathbb{Q})\}_{p\in\mathbb{N}}$ of
the free loop space $\Lambda M$ of $M$ is unbounded. For simply-connected compact manifold, Vigu$\acute{e}$-Poirrier and
Sullivan \cite{ViS} pointed out that the Betti number sequence is bounded if and only if $M$ satisfies
\be
H^*(M; \mathbb{Q})\cong T_{d, n+1}(x)=\mathbb{Q}[x]/(x^{n+1}=0)\nonumber
\ee
with a generator $x$ of degree $d\geq 2$ and height $n+1\geq 2$, where $\dim M = dn$.
Thereafter, the main interest focuses on the compact globally symmetric spaces of rank 1,
where the Gromoll-Meyer assumption does not hold. They consist in
\begin{gather}
 S^n,\quad \mathbb{R}P^n,\quad \mathbb{C}P^n,\quad \mathbb{H}P^n \quad and \quad CaP^2.\nonumber
\end{gather}
A very famous conjecture states that there are infinitely many distinct closed geodesics on every compact simply-connected
Riemannian manifold. The answer is affirmative on $S^2$. See  \cite{Ban93}, \cite{Fra92} and \cite{Hin93} for details.
On the $n$-dimensional Finsler spheres,
however, one gets a negative answer due to the irreversible metrics given by Katok \cite{Kat73},
in which one could only obtain $2[\frac{n+1}{2}]$ distinct
closed geodesics. These explicit metrics reveal the importance of the symmetry.

The multiplicity of closed geodesics on Finsler spheres has been extensively studied by many authors in the last two decades.
In 2004, Bangert and Long \cite{BL10} (published in 2010) proved there are at least two geometrically distinct closed geodesics
on every Finsler two-sphere. Their argument is based on the Maslov-type index theory for symplectic paths.
Since then a great number of results on the multiplicity
of closed geodesics on simply connected Finsler manifolds have appeared, for which we refer readers to
\cite{DuL1}-\cite{DLW2}, \cite{HiR}, \cite{LoD},  \cite{Rad04}-\cite{Rad10},\cite{Tai1}, \cite{Wan1}-\cite{Wan2} and the references therein.

%It is an interesting and highly nontrivial question whether it still holds without any bumpy conditions.
Besides many works on closed geodesics in the literature which study closed geodesics on simply connected
manifolds, we are aware of not many papers on the multiplicity of closed geodesics on non-simply connected manifolds
whose free loop space possesses bounded Betti number sequence, at least when they are endowed with Finsler metrics. For example,
Ballman et al. \cite{BTZ81} proved in 1981 that every Riemannian manifold with the fundamental group being a nontrivial finitely cyclic group
and possessing a generic metric has infinitely many distinct closed geodesics. In 1984, Bangert and Hingston \cite{BH84} proved that any
Riemannian manifold with the fundamental group being an infinite cyclic group has infinitely many distinct closed geodesics.

In order to apply Morse theory to the multiplicity problem of closed geodesics, motivated by the studies on the simply
connected manifolds, in particular, the resonance identity proved by Rademacher \cite{Rad89},
Xiao and Long \cite{XL15} in 2015 investigated the topological structure of the non-contractible
loop space and established the resonance identity for the non-contractible
closed geodesics on $\mathbb{R}P^{2n+1}$ by use of
$\Z_2$ coefficient homology. As an application, Duan, Long and Xiao \cite{DLX15}
proved the existence of at least two distinct non-contractible closed geodesics on $\R P^{3}$ endowed with a bumpy
and irreversible Finsler metric. Subsequently in \cite{Tai16}, Taimanov
used a quite different method from \cite{XL15} to compute the rational equivariant cohomology of
 the non-contractible loop spaces in compact space forms $S^n/ \Gamma$ and
proved the existence of at least two distinct non-contractible closed geodesics on $\mathbb{R}P^2$
endowed with a bumpy irreversible Finsler metric, and there are at least two non-contractible closed geodesics of class $[h]$
on compact space form $M=S^n/ \Gamma$ if $\Gamma$ is an abelian group, $h$ has an even order and is nontrivial in $\pi_1(M)$
and $\pi_1(\Lambda_{h} (M))_{SO(2)}\neq 1$. Then in \cite{Liu17}, Liu combined Fadell-Rabinowitz index theory with Taimanov's topological results
to get many multiplicity results of non-contractible closed geodesics on positively curved Finsler $\mathbb{R}P^n$.
In \cite{LX},  Liu and Xiao established
the resonance identity for the non-contractible closed geodesics on $\mathbb{R}P^n$, and
together with \cite{DLX15} and \cite{Tai16} proved the existence of at least two distinct
non-contractible closed geodesics on every bumpy $\mathbb{R}P^n$ with $n\geq2$.
Furthermore, Liu, Long and Xiao \cite{LLX18} established the resonance identity for non-contractible closed geodesics of
class $[h]$ on compact space form $M=S^n/ \Gamma$ and obtained  at least two non-contractible closed geodesics of
class $[h]$ provided $\Gamma$ is abelian and $h$ is nontrivial in $\pi_1(M)$. Recently, Liu \cite{Liu19} obtained an optimal
lower bound estimation of the number of contractible closed geodesics on bumpy Finsler compact space form $S^{2n+1}/ \Gamma$ with
reversibility $\lambda$ and flag curvature $K$ satisfying $\left(\frac{\lambda}{1+\lambda}\right)^2<K\le 1$.

In this paper, we shall use the Fadell-Rabinowitz index theory and Taimanov's topological results to study the multiplicity
of non-contractible closed geodesics of nontrivial class $[h]$ on positively curved Finsler compact space form $M=S^n/ \Gamma$.
The main results are read as follows:
\begin{theorem} \label{T:main-1}
{\it Let $M=(S^n/ \Gamma, F)$ be an irreversible Finsler compact space form with  $\Gamma$ being an abelian group and $h$ be a nontrivial
element of finite order $p$ in $\pi_1(M)$. If the reversibility $\lambda$ and the flag curvature $K$ satisfy
$\frac{4p^2}{(p+1)^2} \big(\frac{\lambda}{\lambda+1} \big)^2 < K \leq 1$
with $\lambda< \frac{p+1}{p-1}$, then there exist at least $n-1$ non-contractible closed geodesics of class $[h]$.}
\end{theorem}

\begin{theorem} \label{T:main-2}
{\it Let $M=(S^n/ \Gamma, F)$ be an irreversible bumpy Finsler compact space form with  $\Gamma$ being an abelian group and $h$ be a nontrivial
element of finite order $p$ in $\pi_1(M)$. If the reversibility $\lambda$ and the flag curvature $K$ satisfy
$(\frac{4p}{2p+1})^2 (\frac{\lambda}{\lambda+1})^2 < K \leq 1$ with $\lambda<\frac{2p+1}{2p-1}$,	
then there exist at least $2[\frac{n+1}{2}]$ non-contractible closed geodesics of class $[h]$.}
\end{theorem}

\begin{theorem} \label{T:main-3}
{\it Let $M=(S^n/ \Gamma, F)$ be an irreversible Finsler compact space form with  $\Gamma$ being an abelian group and $h$ be a nontrivial
element of finite order $p$ in $\pi_1(M)$. If the reversibility $\lambda$ and the flag curvature $K$ satisfy
$\delta \leq  K \leq 1$ where $(\frac{\lambda}{\lambda+1})^2  <\delta$ if $n$ is odd, then there exists at least
$E(\frac{\sqrt{\delta} (\lambda+1)}{2\lambda } (n-1) )$ non-contractible closed geodesics of class $[h]$,
where $E(a)=\min\{m\in \mathbb{Z}\mid m\geq a\}$. }
\end{theorem}

\begin{remark}
All closed geodesics we obtained are homotopic to $h$. Note that closed geodesics with another homotopic types may be iterations of what we obtained.
Indeed, that is what we have known in the Katok's metrics of \cite{Kat73}. When $n$ is even, the only non-trivial group acting on $S^n$ freely and
isometrically is $\mathbb{Z}_2$ and in this case the compact space form is actually the real projective spaces $\mathbb{R} P^n$ and our results
coincides with \cite{Liu17}.

We just make some more remarks about our results. Theorem \ref{T:main-1} and Theorem \ref{T:main-3} are inspired by Rademacher \cite{Rad94}
and \cite{Rad07} in which he used Fadell-Rabinowitz index in a relative version to study the closed geodesic problems on compact simply connected
manifolds, and also by Liu \cite{Liu17} in which he used an absolute version of the Fadell-Rabinowitz index to study the multiplicity of
non-contractible closed geodesics on positively curved Finsler $\mathbb{R}P^n$, note that Ballman et al. \cite{BTZ82} applied
Lusternik-Schnirelmann theory to obtained the multiplicity of  closed geodesics on positively curved Riemannian $S^n$.
In contrast to \cite{BTZ82}, \cite{Rad94} and \cite{Rad07}, we are concerned with the number of non-contractible closed geodesics of
nontrivial class $[h]$ on non-simply connected manifolds $S^n/\Gamma$, our proof uses an absolute version of the
Fadell-Rabinowitz index which we restrict it to the non-contractible component of the free loop space in which the curves are homotopic to $h$.

In 1973, Katok \cite{Kat73} found some non-symmetric Finsler metrics on $S^n$ with only finitely many prime closed geodesics.
The smallest number of closed geodesics on $S^n$ that one obtains in these examples is $2[\frac{n+1}{2}]$ (cf. \cite{Zil83}),
which implies that our Theorem \ref{T:main-2} gives optimal lower bound estimation of the number of non-contractible closed geodesics of class $[h]$.
The crucial point of the proof is the $S^1$-equivariant Poincar$\acute{e}$ series of the non-contractible component of the free loop space in
which the curves are homotopic to $h$ derived by Taimanov \cite{Tai16}.
\end{remark}

We obtain the following generic results about the existence of infinitely many distinct non-contractible closed geodesics.
\begin{theorem} \label{T:main-4}
	There are infinitely many distinct non-contractible closed geodesics of nontrivial class $[h]$ on compact Finsler space form $(S^n/\Gamma,F)$ for $C^4$-generic metrics for $n \geq 3$, if the flag curvatures satisfy
	\be
	\frac{\lambda^2}{(\lambda+1)^2} < K \leq 1, \;\; \text{ if $n$ is odd, }  \text{ or }
	(\frac{\lambda}{\lambda +1})^2 \frac{4}{(n-1)^2} < K \leq 1, \;\; \text{ if $n$ is even. }
	\ee
\end{theorem}

\begin{remark}
In contrast to \cite{Rad94}, we obtain $C^4$-generic results instead of $C^2$ ones since we work with osculating Riemannian metrics defined in \eqref{E:Osc-R} as considering perturbations of Finsler metrics.
	
\end{remark}

This paper is organized as follows. In section \ref{S:Morse}, we review the necessary backgrounds about Morse theory on the loop space.
In section \ref{S:FR}, we apply the Fadell-Rabinowitz index to non-contractible component of the free loop space in which the curves
are homotopic to $h$, and in section \ref{S:Multi} we estimate the Morse index of non-contractible closed geodesics of nontrivial class $[h]$
with suitable pinching conditions and prove Theorem \ref{T:main-1}, \ref{T:main-2} and \ref{T:main-3} of this paper.
In the last section, we prove there are infinitely many distinct non-contractible closed geodesics of nontrivial class $[h]$ on non-simply connected manifolds $S^n/\Gamma$ with $C^4$-generic Finsler metrics.

\section{Morse theory for closed geodesics} \label{S:Morse}

A closed geodesic $\gamma:S^1 =\mathbb{R}/\mathbb{Z} \to M$ on a Finsler manifold $(M,F)$ is a critical point of the energy functional
\be \label{E:energy-E}
E(\gamma)=\frac{1}{2}\int_{S^1}F(\gamma(t),\dot{\gamma}(t))^2 dt\nonumber
\ee
on the free loop space $\Lambda M:= H^1(S^1,M)$ with positive energy. This energy functional is $C^{1,1}$ and invariant under $S^1$-action
\be
(s\cdot\gamma)(t)=\gamma(t+s), \forall t,s \in S^1.\nonumber
\ee
Along a closed geodesic $c$, one could define the index form as well and denote the Morse index by $i(c)$ and the nullity by
%= \sharp \{ \text{ the maximal negative dimension } of E^{''}(c) \}
$\nu(c) = \dim(ker E^{''}(c))-1$. In particular, if $c$ is non-degenerate, one has $\nu(c) =0$. As we mentioned is section \ref{S:Intro},
$c^m, m \geq 2$ are also closed geodesics, thus these indices are also well defined on $c^m$. Therefore,
we define the average index of closed geodesic $c$ by
\be
\hat{i}(c)=\lim_{m\to\infty}\frac{i(c^m)}{m},\nonumber
\ee
and its mean average index by
\be
\alpha(c) = \frac{\widehat{i}(c)}{L(c)},
\ee
where $L(c)$ is the length of $c$ defined in (1.1).
%The sub-level sets of $E$ are denoted by\be \label{E:Sub-set}\Lambda^{\kappa}=\{d \in \Lambda M\;|\;E(d)\le\kappa\},\quad \Lambda^{\kappa-}=\{d\in \Lambda M\;|\; E(d)<\kappa\},\quad \forall \kappa \ge 0.\nonumber\eeIn particular, we denote the sub-critical set of a closed geodesic $c$ by $\Lambda(c)=\{\gamma \in\Lambda \mid E(\gamma)<E(c)\}$.Using singular homology with rational coefficients, we consider the following critical $\mathbb{Q}$-module of closed geodesic $c\in\Lambda M$:\be \overline{C}_*(E,c)= H_*\left((\Lambda(c)\cup S^1\cdot c)/S^1,\Lambda(c)/S^1; \mathbb{Q}\right).\ee

In the following we let $M=S^n/ \Gamma$ and $h$ be a nontrivial element of finite order $p$ in
$\pi_1(M)$, where the integer $n\geq2$,
$\Gamma$ is a finite abelian group which acts freely and isometrically on the $n$-sphere and therefore
$M$ is diffeomorphic to a compact space form. Then the free loop space $\Lambda M$ possesses
a natural decomposition
$$ \Lambda M=\bigsqcup_{g\in \pi_1(M)}\Lambda_g M,  $$
where $\Lambda_g M$ is the connected component of $\Lambda M$ whose elements are homotopic
to $g$. We set $\Lambda_{h}(c) = \{\gamma\in \Lambda_{h}M\mid E(\gamma)<E(c)\}$.
Note that for a non-contractible minimal closed geodesic $c$ of class $[h]$, $c^m\in\Lambda_h M$
if and only if $m\equiv 1(\mod~ p)$ since $\Gamma$ is an abelian group,
where a closed geodesic $c$ of the class $[h]$ is called {\it minimal} if it is not an iteration of any other
closed geodesics in class $[h]$, note that a minimal closed geodesic $c$ of the class $[h]$ may not be prime.

We call a non-contractible minimal closed geodesic $c$ of class $[h]$ satisfies the isolation
condition, if the following holds:

{\bf (Iso)  For all $m\in \mathbb{N}$ the orbit $S^1\cdot c^{p(m-1)+1}$ is an
isolated critical orbit of $E$. }

Note that if the number of non-contractible minimal closed geodesics of class $[h]$ on $M$
is finite, then all the non-contractible minimal closed geodesics of class $[h]$ satisfy (Iso).

Now we restrict the energy functional $E$ on the non-contractible component $\Lambda_h M$ and study the Morse theory on $\Lambda_h M$.
Thus we define the $S^1$-critical modules of $c^{p(m-1)+1}$ for $E|_{\Lambda_h M}$ as
$$ \overline{C}_*(E,c^{p(m-1)+1}; [h])
   = H_*\left((\Lambda_h(c^{p(m-1)+1})\cup S^1\cdot c^{p(m-1)+1})/S^1,\Lambda_h(c^{p(m-1)+1})/S^1; \mathbb{Q}\right).$$
Following \cite{Rad92}, Section 6.2, we can use finite-dimensional approximations to $\Lambda_h M$ to apply the results of
D. Gromoll and W. Meyer \cite{GM1969Top} to a given non-contractible closed geodesic $c$ of class $[h]$ which is isolated as a critical orbit.
Then we have

\begin{proposition} \label{P:dim-1}
Let $k_j(c)\equiv\dim\overline{C}_j(E,c; [h])$. Then $k_j(c)$  equal to $0$ when $j<i(c)$ or $j>i(c)+\nu(c)$. $k_j(c)$ is
either $0$ or $1$ when $j = i(c)$ or $j=i(c) + \nu(c)$. If $\nu(c)=0$, i.e., $c$ is non-degenerate, then $k_j(c)=1$ holds only for $j=i(c)$.
\end{proposition}

An important ingredient in our proof of Theorem 1.2 is the following result due to Taimanov \cite{Tai16}.
\begin{lemma}[Theorem 3 of \cite{Tai16} or Lemma 2.3 of \cite{LLX18}] \label{L:Betti-Tai}
Let $\Gamma$ be a finite abelian group. For $M=S^n/\Gamma$, and a nontrivial $h \in \pi_1(M)$ we have
\begin{enumerate}
\item When $n=2k+1$ is odd, the $S^1$-cohomology ring of $\Lambda_h M$ has the form
$$H^{S^1, *}(\Lambda_h M; \mathbb{Q})=\mathbb{Q}[w, z]/ \{w^{k+1} = 0\}, \quad deg(w)=2, \quad deg(z)=2k$$
Then the $S^1$-equivariant Poincar$\acute{e}$ series
of $\Lambda_h M$ is given by
\begin{equation}
\begin{split}
P^{S^1}(\Lambda_h M; \mathbb{Q})(t) & =\frac{1-t^{2k+2}}{(1-t^2)(1-t^{2k})} \\
&=\frac{1}{1-t^2}+\frac{t^{2k}}{1-t^{2k}}\\
&=(1+t^2+t^4+\cdots+t^{2k}+\cdots)+(t^{2k}+t^{4k}+t^{6k}+\cdots),
\end{split}\nonumber
\end{equation}
which yields Betti numbers
\be
b_q = \mathrm{rank} H_q^{S^1}(\Lambda_h M;\mathbb{Q}) = \begin{cases}
    2,&\quad {\it if}\quad q\in \{j(n-1)\mid j\in\mathbb{N} \} \\
    1,&\quad {\it if}\quad q\in2\mathbb{N}_0\setminus \{j(n-1)\mid j\in\mathbb{N}\} \\
    0, &\quad {\it otherwise}.
    \end{cases}
\ee

\item When $n=2k$ is even, the $S^1$-cohomology ring of $\Lambda_h M$ has the form
$$H^{S^1, *}(\Lambda_h M; \mathbb{Q})=\mathbb{Q}[w, z]/ \{w^{2k} = 0\}, \quad deg(w)=2, \quad deg(z)=4k-2$$
Then the $S^1$-equivariant Poincar$\acute{e}$ series
of $\Lambda_h M$ is given by
\be
\begin{split}
P^{S^1}(\Lambda_h M; \mathbb{Q})(t) & =\frac{1-t^{4k}}{(1-t^2)(1-t^{4k-2})} \\
 & = \frac{1}{1-t^2}+\frac{t^{4k-2}}{1-t^{4k-2}}\\
&= (1+t^2+t^4+\cdots+t^{2k}+\cdots)\\ &~~~~~~~~~~~~~~~+(t^{4k-2}+t^{2(4k-2)}+t^{3(4k-2)}+\cdots),
\end{split}\nonumber
\ee
which yields Betti numbers
\be
b_q = \mathrm{rank} H_q^{S^1}(\Lambda_h M;\mathbb{Q})
= \begin{cases}
    2,&\quad {\it if}\quad q\in \{2j(n-1)\mid j\in\mathbb{N}\}, \\
    1,&\quad {\it if}\quad q\in2\mathbb{N}_0\setminus \{2j(n-1)\mid j\in\mathbb{N}\}, \\
    0, &\quad {\it otherwise}.
\end{cases}
\ee
\end{enumerate}
\end{lemma}

Note that for a non-contractible minimal closed geodesic $c$ of class $[h]$, $c^m\in\Lambda_h M$
if and only if $m\equiv 1(\mod~ p)$.
Thus we have the Morse inequality for non-contractible closed geodesics of class $[h]$:
\begin{lemma}[Theorem I.4.3 of \cite{Chan93}]
Assume that $M=S^n/\Gamma$ be a Finsler manifold with finitely many non-contractible minimal closed geodesics of class $[h]$,
denoted by $\{c_j\}_{1\le j\le k}$. Set
\begin{align}
M_q =\sum_{1\le j\le k,\; m\ge 1}\dim{\overline{C}}_q(E, c^{p(m-1)+1}_j; [h]),\quad q\in\mathbb{Z}.
\end{align}
Then for every integer $q \ge 0$ there holds
\bea
M_q \ge {b}_q. \eea
\end{lemma}

\section{Fadell-Rabinowitz index for non-contractible closed geodesics} \label{S:FR}

%In this section, we use singular homology modules with $\mathbb{Q}$-coefficients and apply the Fadell-Rabinowitz index to the non-contractible component $\Lambda_h M$ of the free loop space on compact space form $S^{n}/\Gamma$. It will be the main ingredients for proving the main results.
In 1978, Fadell and Rabinowitz \cite{FR78} introduced a topological index theory by using the equivariant Borel
cohomology. Via a classifying map the cohomology ring $H^*(BG)$ of a classifying space is a subring of the equivariant cohomology
$H_G^*(X)$ of the G-space $X$. For a characteristic class $\eta\in H^*(BG)$ the Fadell-Rabinowitz index $index_\eta X$ of X
is the smallest non-negative integer k with $\eta^k=0$ in $H_G^*(X)$. Hence the classes $\eta^j$
for $j=0, \ldots, index_\eta X- 1$ define a sequence of subordinate cohomology classes.
Fadell and Rabinowitz used this index to study the bifurcation of time periodic solutions
from an equilibrium solution for Hamiltonian systems of ordinary differential
equations. Ekeland and Hofer in \cite{EH87} of 1987 used this index for the group $S^1$ to
derive a close relationship between the set of Maslov-type indices of closed characteristics on convex
Hamiltonian energy surfaces in $\R^{2n}$ and the set of even
positive integers, which is the core in studying multiplicity and ellipticity of
closed characteristics on compact convex hypersurfaces (cf. \cite{LZ02}). Rademacher in \cite{Rad94} of 1994 developped
the relative version of Fadell-Rabinowitz index to study the existence of closed geodesics on compact simply connected
manifolds, which played an important role in studying multiplicity and stability of
closed geodesics on Finsler spheres, cf. \cite{Rad07}, \cite{Wan13}.

In this section, we use only singular homology modules with $\mathbb{Q}$-coefficients and apply the Fadell-Rabinowitz
index in an absolute version to the non-contractible component $\Lambda_h M$ of the free
loop space on $M=S^n/\Gamma$, which is the main ingredient for proving our Theorem 1.1 and Theorem 1.3.
The main idea is similar with previous work \cite{Liu17}, which is  to extend the work of Rademacher \cite{Rad94} to
the non-contractible component $\Lambda_h M$.
Thus some of the following lemmas and theorems may look like those of  \cite{Liu17} and \cite{Rad94}.
However, we note that the topological structure of $S^n/\Gamma$ is actually very different to $\mathbb{R}P^n$ when $n$ is
odd and thus our results are new.

It is well known that, due to \cite{LF51} for example, there always exists a non-contractible closed geodesic $c$ on $(S^n/\Gamma,F)$,
which is the minimal point of $E$ on $\Lambda_h M$. Let $E(c) = \frac{\kappa_1^2}{2}$ for $\kappa_1 >0$. For any $\kappa > 0$, we denote
\be
\Lambda_h^\kappa := \{ \gamma \in \Lambda_h M \big| E(\gamma) \leq \frac{1}{2} \kappa^2\}
\ee
We define the function
\be
d_h(\kappa) = index_{\eta} \Lambda_h^\kappa : [\kappa_1,+\infty] \to \mathbb{N}\cup \{\infty\}
\ee
where $H^*(BS^1) = H^*(CP^\infty)= \mathbb{Q}[\eta]$.

\begin{definition} \label{D:G-Index}
The lower index $\underline{\sigma}$ and upper index $\overline{\sigma}$ are given by
\be
\underline{\sigma} := \liminf_{\kappa \to \infty} \frac{d_h(\kappa)}{\kappa}, \quad
\overline{\sigma}:=\limsup_{\kappa \to \infty} \frac{d_h(\kappa)}{\kappa},
\ee
and $[\underline{\sigma},\overline{\sigma}]$ is called the global index interval.
\end{definition}

Similarly to Proposition 5.2 of \cite{Rad94}, we have the following estimates for the global index interval
which are extentions of Lemma 3.2 of \cite{Liu17}.
\begin{lemma} \label{L:Cur}
Let $K$ and $Ric$ be the flag curvature and Ricci curvature of the manifold $(S^n/\Gamma,F)$ respectively. Then we have
\begin{itemize}
\item If $K \leq \Delta^2$, then $\overline{\sigma} \leq \frac{\Delta (n-1)}{2 \pi}$.

\item If $K \geq \delta^2$, then $\overline{\sigma} \geq \frac{\delta(n-1)}{2 \pi}$.

\item If $Ric \geq \delta^2(n-1)$, then $\underline{\sigma} \geq \frac{\delta}{2\pi}$.
\end{itemize}
\end{lemma}

\begin{proof}
The proof mostly follows \cite{Rad94}, for the reader's convenience, we sketch a proof here. We only prove (a)
by the Morse-Schoenberg comparison theorem.

Choose a sequence $(F_i)$ of bumpy metrics converging to $F$ in the strong $C^2$ topology. Let
\[
E_i(\gamma)=\frac{1}{2}\int_{S^1}F_i(\gamma(t), \dot{\gamma}(t))^2dt
\]
be the energy functional and $\Lambda_{h,i}^\kappa=\{\gamma\in \Lambda_h M\;|\; E_i(\gamma)\leq\frac{1}{2} \kappa^2\}$.
 Choose $(a_i)\subset \mathbb{R}$ with $a=\lim_{i\rightarrow\infty} a_i$ and $\Lambda_h^a\subset\Lambda_{h,i}^{a_i}$ for all $i$. We can
choose a sequence $(\Delta_i)\subset\mathbb{R}$ such that the flag curvature of $F_i$ is bounded from
above by $\Delta_i^2$ and $\Delta=\lim_{i\rightarrow\infty} \Delta_i$.

It follows from the Morse-Schoenberg comparison
theorem that the index $i(c)$ of a closed geodesic $c$ of the metric $F_i$ with length at most $a_i$ satisfies
\bea i(c)\leq \left(\frac{a_i}{\pi}\Delta_i+1\right)(n-1)=:k_i.\eea
Hence it follows from the Morse lemma that $H^k_{S^1}(\Lambda_{h,i}^{a_i})=0$ for all $k>k_i$.
Since $\Lambda_h^a\subset\Lambda_{h,i}^{a_i}$ it follows from the composition
\be
H^k(BS^1)\rightarrow H_{S^1}^k(\Lambda_{h,i}^{a_i})\rightarrow  H_{S^1}^k(\Lambda_h^{a}).\nonumber
\ee
of restriction homomorphisms that $\eta^k=0$ in $H_{S^1}^{2k}(\Lambda_h^{a})$ for all $k>\frac{k_i}{2}$.
Hence we have $d_h(a)\leq \frac{k_i}{2}$ by (3.2), which together with (3.3)-(3.4) gives
$\overline{\sigma}\leq \frac{\Delta(n-1)}{2\pi}$.
\end{proof}

With similar arguments in \cite{Liu17}, we have
\begin{proposition} \label{P:Ind-2}
For any compact Finsler space form $(S^n/\Gamma,F)$, one has $\lim_{\kappa \to \infty} d_h(\kappa) = \infty$.
\end{proposition}
\begin{proof}
For the reader's convenience, we sketch a proof here.
Let $(S^n, N_\alpha)$ denote the sphere $S^n$ endowed with the Katok metric, then it induces
a Finsler metric for $S^n/\Gamma$, which is still denoted by $N_\alpha$ for simplicity. Therefore the natural projection
\bea \pi:(S^n, N_\alpha)\rightarrow (S^n/\Gamma, N_\alpha)\nonumber\eea
is locally isometric. Since $(S^n, N_\alpha)$ and $(S^n/\Gamma, N_\alpha)$ have constant flag curvature 1, then for $(S^n/\Gamma, N_\alpha)$,
$\underline{\sigma}>0$ by Lemma 3.2. Thus for $S^n/\Gamma$ endowed with any Finsler metric $F$, $\underline{\sigma}>0$
also holds which implies $\lim_{\kappa\rightarrow+\infty}d_g(\kappa)=+\infty$ by (3.3). In fact, Let $F,F^*$ be two Finsler metrics on $S^n/\Gamma$ with
\bea F^*(X)^2/D^2\leq  F(X)^2\leq D^2F^*(X)^2\nonumber\eea
for all tangent vector $X$ and some $D>1$. Let $[\underline{\sigma}, \overline{\sigma}]$, $[\underline{\sigma}^*, \overline{\sigma}^*]$
be the global index intervals with respect to the metrics $F,F^*$ respectively. Then by definition we have
\bea D^{-1}\underline{\sigma}\leq \underline{\sigma}^*\leq D\underline{\sigma},
\quad D^{-1}\overline{\sigma}\leq\overline{\sigma}^*\leq D\overline{\sigma}\nonumber\eea
which implies our desired result.\end{proof}

\begin{proposition}
The function $d_h:  [\kappa_1, +\infty)\rightarrow \mathbb{N}$ is non-decreasing and $\lim_{\lambda\searrow \kappa}d_h(\lambda)=d_h(\kappa)$.
For each discontinuous point $\kappa$ of $d_h$, $\frac{1}{2} \kappa^2$ is a critical value of the energy functional $E$ on $\Lambda_h$.
In particular, if $d_h(\kappa)-d_h(\kappa-)\ge 2$, then there are infinitely many minimal non-contractible closed geodesics of class $[h]$ with
energy $\frac{1}{2} \kappa^2$, where $d_h(\kappa-)=\lim_{\epsilon\searrow 0}d_h(\kappa-\epsilon)$, $t\searrow a$  means $t>a$ and $t \to a$.
\end{proposition}
\begin{proof}
We have that $d_h(\lambda)$ is finite by the proof of Lemma 3.2. Noticing that the definitions in (3.1)-(3.2),
then the rest of our proof is completely similar to Lemmas 5.5-5.6 and Corollary 5.7 of \cite{Rad94}, thus
we omit it.
\end{proof}

For each $i \geq 1$, define
\be
\kappa_i = \inf\{\delta \mid d_h(\delta) \geq i \}
\ee
which is well-defined due to Proposition \ref{P:Ind-2}. Then as Theorem 3.5 and Theorem 3.6 of \cite{Liu17}, we have
\begin{theorem} \label{T:Strict-in}
Suppose there are only finitely many minimal non-contractible closed geodesics of class $[h]$ on $(S^n/\Gamma, F)$.
Then each $\frac{1}{2} \kappa_i^2$ is a critical value of $E$ on $\Lambda_h$. If $\kappa_i=\kappa_j$ for some $i<j$,
then there are infinitely many minimal non-contractible closed geodesics of class $[h]$ on $(S^n/\Gamma, F)$.
\end{theorem}

%\begin{proof}
%We replace Lemma 2.2 used in the proof of Lemma 2.3 of \cite{Wan13} by the above Lemma 3.4, then our proof is similar to that of Lemma 2.3 of \cite{Wan13}.
%\end{proof}

\begin{theorem} \label{T:index}
Suppose there are only finitely many minimal non-contractible closed geodesics of class $[h]$ on $(S^n/\Gamma, F)$.
Then for every $i\in\mathbb{N}$, there exists a non-contractible closed geodesic $c$ of class $[h]$ on $(S^n/\Gamma, F)$ such that
\be \label{E:0}
E(c)=\frac{1}{2} \kappa_i^2,\quad \overline{C}_{2i-2}(E, c;[h])\neq 0.
\ee
\end{theorem}

\begin{proof}
Choose $\epsilon$ small enough such that the interval
$(\frac{1}{2} (\kappa_i-\epsilon)^2,\,\frac{1}{2} (\kappa_i+\epsilon)^2)$ contains no critical values of $E|_{\Lambda_h M}$ except
$\frac{1}{2} \kappa_i^2$, where $M=S^n/\Gamma$.
By the same proof of (2.12) of \cite{Wan13}, we have
\be \label{E:1}
 H^{\ast}_{S^1}(\Lambda_h^{\kappa_i+\epsilon},\;\Lambda_h^{\kappa_i-\epsilon})
\cong H^\ast(\Lambda_h^{\kappa_i+\epsilon}/S^1,\;\Lambda_h^{\kappa_i-\epsilon}/S^1).
\ee
Also as (2.13) of \cite{Wan13}, we have
\be \label{E:2}
H_\ast(\Lambda_h^{\kappa_i+\epsilon}/S^1,\;\Lambda_h^{\kappa_i-\epsilon}/S^1)
\cong H^\ast(\Lambda_h^{\kappa_i+\epsilon}/S^1,\;\Lambda_h^{\kappa_i-\epsilon}/S^1).
\ee
Then we combining \eqref{E:1}-\eqref{E:2} with Theorem 1.4.2 of \cite{Chan93} to obtain
\be \label{E:3}
H_{S^1}^*(\Lambda_h^{\kappa_i+\epsilon},\;\Lambda_h^{\kappa_i-\epsilon})
=\bigoplus_{E(c)=\frac{1}{2} \kappa_i^2}\overline{C}_{\ast}(E, \; c;[h]).
\ee
For an $S^1$-space $X$, we denote by
$X_{S^1}$ the homotopical quotient of $X$ by $S^1$, i.e., $X_{S^1}=S^\infty\times_{S^1}X$,
where $S^\infty$ is the unit sphere in an infinite dimensional complex Hilbert space.
Similar to P.431 of \cite{EH87}, we have
\be \label{E:4}
H^{2(i-1)}((\Lambda_h^{\kappa_i+\epsilon})_{S^1},\,(\Lambda_h^{\kappa_i-\epsilon})_{S^1})
\mapright{q^\ast} H^{2(i-1)}((\Lambda_h^{\kappa_i+\epsilon})_{S^1} )
\mapright{p^\ast}H^{2(i-1)}((\Lambda_h^{\kappa_i-\epsilon})_{S^1}),
\ee
where $p$ and $q$ are natural inclusions. Denote by
$f: (\Lambda_h^{\kappa_i+\epsilon})_{S^1}\rightarrow CP^\infty$ a classifying map and let
$f^{\pm}=f|_{(\Lambda_h^{\kappa_i\pm\epsilon})_{S^1}}$. Then clearly each
$f^{\pm}: (\Lambda_h^{\kappa_i\pm\epsilon})_{S^1}\rightarrow CP^\infty$ is a classifying
map on $(\Lambda_h^{\kappa_i\pm\epsilon})_{S^1}$. Let $\eta \in H^2(CP^\infty)$ be the first
universal Chern class.

By definition of $\kappa_i$ in (3.5), we have $d_h(\kappa_i-\epsilon)< i$, hence
$(f^-)^\ast(\eta^{i-1})=0$. Note that
$p^\ast(f^+)^\ast(\eta^{i-1})=(f^-)^\ast(\eta^{i-1})$.
Hence the exactness of \eqref{E:4} yields that there exists a
$\sigma\in H^{2(i-1)}((\Lambda_h^{\kappa_i+\epsilon})_{S^1},\,(\Lambda_h^{\kappa_i-\epsilon})_{S^1})$
such that $q^\ast(\sigma)=(f^+)^\ast(\eta^{i-1})$.
Since $d_h(\kappa_i+\epsilon)\ge i$, we have $(f^+)^\ast(\eta^{i-1})\neq 0$.
Hence $\sigma\neq 0$, and then
$$H_{S^1}^{2i-2}(\Lambda_h^{\kappa_i+\epsilon},\;\Lambda_h^{\kappa_i-\epsilon})=
H^{2i-2}((\Lambda_h^{\kappa_i+\epsilon})_{S^1},\,(\Lambda_h^{\kappa_i-\epsilon})_{S^1})\neq 0. $$
Thus \eqref{E:0} follows from \eqref{E:3}.
\end{proof}

Note that by Proposition \ref{P:dim-1} , \eqref{E:0} yields
\be
E(c)=\frac{1}{2} \kappa_i^2,\quad
i(c)\leq 2i-2 \leq i(c)+\nu(c).
\ee

Similar to Definition 1.4 of \cite{LZ02}, we have

\begin{definition}
A minimal non-contractible closed geodesic $c$ of class $[h]$ is {\it $(p(m-1)+1, i)$-variationally visible}, if there exist some $m, i\in \mathbb{N}$
such that \eqref{E:0} holds for $c^{p(m-1)+1}$ and $\kappa_i$. We call $c$ {\it infinitely variationally visible} if there exist infinitely many $(m, i)\in \mathbb{N}^2$ such that $c$ is $(p(m-1)+1, i)$-variationally visible.
Denote by $\mathcal{V}_\infty(S^n/\Gamma, F;[h])$ the set of infinitely
variationally visible non-contractible closed geodesics of class $[h]$ on $(S^n/\Gamma,F)$.
\end{definition}

Similar to Theorem V.3.15 of \cite{Eke90} and Theorem 6.4 of \cite{Rad94}, we have

\begin{lemma} \label{L:ave-ind}
Suppose there are only finitely many distinct closed geodesics of class $[h]$ on $(S^n/\Gamma, F)$
and $\hat{i}(c)>0$ for any $c\in\mathcal{V}_\infty(S^n/\Gamma, F;[h])$, then we have
\be
\sum_{c\in\mathcal{V}_\infty(S^n/\Gamma, F;[h])}\frac{1}{\hat{i}(c)}\geq \frac{p}{2}.\nonumber
\ee
\end{lemma}
\begin{proof}
Noticing that only $((m-1)p+1)$-iterations of minimal non-contractible closed geodesics of class $[h]$
can be critical points of the energy functional $E$ on $\Lambda_h (S^n/\Gamma)$,  then our proof follows from that of Theorem V.3.15 of \cite{Eke90},
if we replace Proposition 7 used in the proof of Theorem V.3.15 of \cite{Eke90} by the above Theorem \ref{T:index} and (3.11).
\end{proof}

\section{Multiplicity results of non-contractible closed geodesics} \label{S:Multi}
In this section, we give proofs of Theorems 1.1-1.3.
The following lemmas proved by Rademacher are useful for us to control the index of non-contractible closed geodesics.
\begin{lemma}(cf. Lemma 1 of \cite{Rad07}) \label{L:index-up}
Let $c$ be a closed geodesic on a Finsler manifold $(M, F)$ of dimension $n$ with positive flag curvature $K \geq \delta$ for some $\delta>0$, if the length $L(c)$ satisfies $L(c)>\frac{k\pi}{\sqrt{\delta}}$ for some positive integer $k$ then $i(c)\geq k(n-1)$.
\end{lemma}
\begin{lemma}(cf. Theorem 1 and Theorem 4 of \cite{Rad04}) \label{L:length}
Let $M$ be a compact and simply-connected manifold with a Finsler metric $F$ with reversibility $\lambda$ and flag curvature $K$ satisfying
$0< K \leq 1$ resp. $\frac{\lambda^2}{(\lambda+1)^2} < K\leq 1$ if the dimension $n$ is odd. Then the length of a
closed geodesic is bounded from below by $\pi \frac{\lambda+1}{\lambda}$.
\end{lemma}

Now we can estimate the index of non-contractible closed geodesics of class $[h]$ in compact space form $(S^n/\Gamma,F)$.

\begin{lemma} \label{L:index-1}
Let $c$ be a non-contractible closed geodesic of class $[h]$ on Finsler compact space form $(S^n/\Gamma,F)$  with reversibility $\lambda$ and flag curvature $K$ satisfying $\frac{4p^2}{(p+1)^2} \big(\frac{\lambda}{\lambda+1} \big)^2< K \leq 1$ with $\lambda<\frac{p+1}{p-1}$, then $i(c^m) \geq 2(n-1)$ for $m \geq p+1$.
\end{lemma}

\begin{proof}
Since $(S^n,F)$ is the universal covering of $(S^n/\Gamma,F)$ and $c \in \Lambda_h M$, then $c^p$ is a closed geodesic in $(S^n,F)$
where $p$ is the order of $h$. Due to Lemma \ref{L:length}, for $m\geq p+1$ we have
\begin{align}
L(c^m) = \frac{m}{p} L(c^p) \geq \frac{m \pi }{p} \frac{\lambda+1}{\lambda} \geq \frac{p+1}{2p} \frac{\lambda+1}{\lambda} 2 \pi.
\end{align}
By assumption, we choose $\delta > \frac{4p^2}{(p+1)^2} \big(\frac{\lambda}{\lambda+1} \big)^2$ such that $K \geq \delta$, it together
with (4.1) gives  $L(c^m)> \frac{2\pi}{\sqrt{\delta}}$ for $m\geq p+1$ which implies $i(c^m)\geq 2(n-1)$ by Lemma 4.1.
\end{proof}

\begin{lemma} \label{L:index-2}
Let $c$ be a non-contractible closed geodesic of class $[h]$ on Finsler compact space form $(S^n/\Gamma,F)$ with reversibility $\lambda$ and flag curvature $K$ satisfying $(\frac{4p}{2p+1})^2 (\frac{\lambda}{\lambda+1})^2 < K \leq 1$ with $\lambda<\frac{2p+1}{2p-1}$, then $i(c^m) \geq 4(n-1)$ for $m \geq 2p+1$.
\end{lemma}
\begin{proof}
Since $(S^n,F)$ is the universal covering of $(S^n/\Gamma,F)$ and $c \in \Lambda_h M$, then $c^p$ is a closed geodesic in $(S^n,F)$
where $p$ is the order of $h$. Due to Lemma \ref{L:length}, for $m\geq 2p+1$ we have
\begin{align}
L(c^m) = \frac{m}{p} L(c^p) \geq \frac{m \pi }{p} \frac{\lambda+1}{\lambda} \geq \frac{2p+1}{2p} \frac{\lambda+1}{\lambda} 2 \pi.
\end{align}
By assumption, we choose $\delta > \frac{(4p)^2}{(2p+1)^2} \big(\frac{\lambda}{\lambda+1} \big)^2$ such that $K \geq \delta$, it together
with (4.2) gives  $L(c^m)> \frac{4\pi}{\sqrt{\delta}}$ for $m\geq 2p+1$ which implies $i(c^m)\geq 4(n-1)$ by Lemma 4.1.
\end{proof}

\begin{proof}[Proof of Theorem \ref{T:main-1}]
We assume that there exist finitely many minimal non-contractible closed
geodesics of class $[h]$. Then by Theorems 3.5-3.6, there exist non-contractible
closed geodesics $c_i$ such that
\be
E(c_i) = \frac{\kappa_i^2}{2},\quad \overline{C}_{2i-2}(E,c_i;[h]) \neq 0
\ee
and $\kappa_i \neq \kappa_j$ for positive integers $i \neq j$. Combining Proposition 2.1 with (4.3),
we obtain $i(c_i)\leq2i-2\leq i(c_i)+\nu(c_i)$ for any $i\in\mathbb{N}$. Then by Lemma 4.3,
note that only $((m-1)p+1)$-iterations of minimal non-contractible closed geodesics of class $[h]$ can be
critical points of the energy functional $E$ on $\Lambda_h (S^n/\Gamma)$,
we obtain that $c_i$ is minimal for $1\leq i\leq n-1$. Then $c_i\neq c_j$ for $1\leq i\neq j\leq n-1$
since $ E(c_i)=\frac{\kappa_i^2}{2}\neq\frac{\kappa_j^2}{2}= E(c_j)$. Thus we get
$n-1$ non-contractible minimal non-contractible closed
geodesics $\{c_i\mid 1\leq i\leq n-1\}$ of class $[h]$.
\end{proof}

\begin{proof}[Proof of Theorem \ref{T:main-2}]
Denote all the minimal non-contractible closed
geodesics of class $[h]$ by $\{c_i\mid 1\leq i\leq k\}$. Then by Lemma 4.4, we
have $i(c_i^m)\geq 4(n-1)$ for $m\geq 2p+1$ which implies that if $c_i^m$ has contribution to
$M_q$ for $0\leq q\leq 4(n-1)-2$, then $m=1$ or $p+1$ since only $((m-1)p+1)$-iterations of $c_i$ have contribution to
$M_q$. Thus by Proposition 2.1 and (2.4), we obtain
\be
M_{4n-6} + M_{4n-8} + \cdots + M_2 + M_0 \leq 2k
\ee
In the following, we prove in two cases:

(i) $n$ is odd. In this case, by (2.5) of Lemma 2.3 and (2.2) of Lemma 2.2, we have
\bea  M_{4n-6}+M_{4n-8}+\cdots+M_2+M_0\geq b_{4n-6}+b_{4n-8}+\cdots+b_2+b_0=2n+1.\eea
Then $k\geq n+1$ follows from (4.4) and (4.5).

(ii) $n$ is even. In this case, by (2.5) of Lemma 2.3 and (2.3) of Lemma 2.2, we have
\bea  M_{4n-6}+M_{4n-8}+\cdots+M_2+M_0\geq b_{4n-6}+b_{4n-8}+\cdots+b_2+b_0=2n-1.\eea
Then $k\geq n$ follows from (4.4) and (4.6). The proof is complete.
\end{proof}

The following lemma helps us to control the mean index of a closed geodesic:
\begin{lemma}(cf. Lemma 2 of \cite{Rad07}) \label{L:Mean-ind}
Let $c$ be a closed geodesic on a compact, simply-connected Finsler manifold with a non-reversible metric satisfying
$0 < \delta \leq K \leq 1$ if $n$ is even and $\frac{\lambda^2}{(\lambda+1)^2} < \delta \leq K \leq 1$ if $n$ is odd. Then $\hat{i}(c) \geq \frac{\sqrt{\delta}(\lambda+1)}{\lambda}(n-1)$.
\end{lemma}

\begin{proof}[Proof of Theorem \ref{T:main-3}]
Denote by all the minimal non-contractible closed geodesics of class $[h]$ by $\{c_i \mid 1 \leq i \leq k\}$,
Since $(S^n, F)$ is the universal covering of $(S^n, F)$, then $c_i^p$ is a closed geodesic on $(S^n,F)$, where $p$ is the order of $h$.
Due to Lemma \ref{L:Mean-ind}, one has
\be
\hat{i}(c_i^p) \geq \sqrt{\delta}\frac{\lambda+1}{\lambda}(n-1),\nonumber
\ee
which implies
\be
\hat{i}(c_i) \geq \frac{1}{p} \sqrt{\delta}\frac{\lambda+1}{\lambda}(n-1).
\ee
Due to Lemma \ref{L:ave-ind}, we have
\be
\sum_{j=1}^k \frac{1}{\hat{i}(c_j)} \geq \frac{p}{2}.
\ee
Combining (4.7) and (4.8), we have
\be
k \geq \frac{\sqrt{\delta} (\lambda+1)}{2\lambda } (n-1) ,\nonumber
\ee
which implies $k \geq E(\frac{\sqrt{\delta} (\lambda+1)}{2\lambda } (n-1))$ since $k$ is an integer, where $E(a)=\min\{m\in \mathbb{Z}\mid m\geq a\}$.\end{proof}

\section{mean average index of non-contractible closed geodesics and generic properties of Finsler metrics} \label{S:generic}

Recall the definition of average index of closed geodesics given in Section \ref{S:Morse}. As we focus on iterations of non-contractible closed geodesics of class $[h]$, only $((m-1)p+1)$-iterations of them stay in the same homotopic component. Therefore, it should be rigorously defined by
\be
\hat{i}_{[h]}(c) = \lim_{m \to \infty} \frac{i(c^{(m-1)p +1})}{(m-1)p+1}.\nonumber
\ee
Note that its value does no depend on the choice of $p \geq 1$. We still use $\hat{i}(c)$ for brevity.

There is a long-standing conjecture that there are infinitely many geometrically distinct closed geodesics on every compact Riemannian manifolds with its dimension being at least two. The Katok's metrics,  which are irreversible and bumpy Finsler metrics, show that the conjecture is not true for Finsler manifolds. Here we point out that they are actually rare and not persisted under aribitary small $C^r$ perturbations for $4 \leq r \leq \infty$.

%Let us recall the definition of average index given in Section \ref{S:Morse}. As we noted in Section \ref{S:FR}, only $(m-1)p+1$ iterations of a non-contractible closed geodesic of class $[h]$ persist within the same homotopic class. Then the average index should be defined as

Recall the mean average index, defined in (2.1) as
\be
\alpha(c) = \frac{\hat{i}(c)}{L(c)} = \frac{i(c)}{\sqrt{2 E(c)}}.
\ee
In Section \ref{S:FR}, we have seen that the jump points of Fadell-Rabinowitz index correspond to critical points of the energy functional. The following lemma can be seen as a refinement on them.

\begin{proposition} \label{P:mai}
Recall the global index interval defined in Definition \ref{D:G-Index}. For all $t \in [\underline{\sigma},\overline{\sigma}]$ then there is a sequence of closed geodesics $c_i \in \Lambda_h M$ with
	\be
	t = \frac{1}{2}\lim_{i \to \infty} \alpha(c_i).\nonumber
	\ee
\end{proposition}
Then it is straightforward to obtain:
\begin{corollary}
Suppose that there are only finitely many closed geodesics of class $[h]$ on $(M,F)$. Then $\underline{\sigma} =\overline{\sigma}:=\sigma$.
\end{corollary}

\begin{remark}
The global index interval is well-defined. Indeed, let us consider a discontinuous sequence of $d_h(\kappa)$ by $\{\kappa_i \}_{i=1}^\infty$. A similar proof as Lemma 6.1 in \cite{Rad94} could show that
\be \label{E:Ind-i}
\underline{\sigma} = \liminf_{i \to \infty} \frac{i}{\kappa_i},\;\; \overline{\sigma} = \limsup_{i \to \infty} \frac{i}{\kappa_i}.
\ee
However, being aware of our considerations focus on non-contractible closed geodesics, we need to prove
the jump of Fadell-Rabinowitz index is bounded from above by $n$ as follow.   	
\end{remark}

\begin{lemma}
	For all $\kappa >0$, we have the upper bound for the jump of $d_h(\kappa)$,
	\be
	d_h(\kappa)-d_h(\kappa-) \leq n.\nonumber
	\ee
\end{lemma}
\begin{proof}
It is trivial for continuous points of $d_h(\kappa)$. If $\kappa>0$ is a discontinuous point of $d_h(\kappa)$, there are critical points of $E$, denoted by $Cr_h(\kappa)$, in the class of $\Lambda_h M$ with energy $\frac{1}{2} \kappa^2$. Firstly, due to the dimensional inequality given in \cite{FR78} and $\eta \in H^2(M)$ is the first Chern class, we have
\be
2(index_\eta (Cr_h(\kappa)) -1) \leq {\rm dim} (Cr_h(\kappa)/S^1).\nonumber
\ee
Therefore, if one has $d_h(\kappa)-d_h(\kappa-)>n$, then  $index_\eta (Cr_h(\kappa)) \geq n+1$, it immediately follows ${\rm dim} (Cr_h(\kappa) /S^1)  \geq 2n$.

On the other hand, $\gamma \in Cr_h(\kappa)$ are periodic orbits of geodesic flows defined on the spherical bundles of $M$ with the same periods (may not be the minimal one) and energy $\frac{1}{2} \kappa^2$. We note that any period orbit $\gamma \in Cr_h(\kappa)$  is one to one corresponding to its initial data $\gamma(0) \in  Cr_h(\kappa) /S^1$. They are all located on the energy surface $E =\frac{1}{2} \kappa^2$ whose dimension is up to $2n-2$ since the dimension of the phase space is $2n-1$. It is a contradiction.
\end{proof}

Since  $\nu(c) \leq 2n-2$,  by using Theorem \ref{T:index} and Definition 3.7 we have the following Lemma:
\begin{lemma} \label{L:mean-con}
		If $c$ is $(p(m-1)+1, i)$-variationally visible, the mean average index $\alpha(c)$ satifies
		\be \label{E:mai-c}
		\big| \alpha(c) -  \frac{2(i-1)}{\kappa_i} \big| \leq \frac{3(n-1)}{\kappa_i}
		\ee
\end{lemma}

\begin{proof}
	It follows from the definition of $(p(m-1)+1,i)$-variationally visible that there exists $m \geq 1$ such that $|i(c^{(m-1)p+1})-2(i-1)| \leq 2n-2$. Note the average index inequality (see \cite{Rad89}) guarantees
	\be
	| i(c^{(m-1)p+1})- ((m-1)p+1)\hat{i}(c)| \leq n-1.\nonumber
	\ee
	We have
	\be
	|((m-1)p+1)\hat{i}(c) - 2(i-1)| \leq 3n-3.\nonumber
	\ee 	
	Dividing $\kappa_i$ on both hand sides and use the fact that
	$L(c^{(m-1)p +1}) = ((m-1)p+1)L(c) = \kappa_i$, we arrive the conclusion.
\end{proof}

\begin{proof}[Proof of Proposition \ref{P:mai}]

It is sufficient to show that
\be
[\underline{\sigma},\overline{\sigma}] \subset  \overline{A:=\{ \frac{1}{2} \alpha(c_i)| i \in \BBN \}},\nonumber
\ee
where $c_i$ is $(p(m-1)+1,i)$-variationally visible for some $i \in \BBN$.

If $\underline{\sigma} = \overline{\sigma}$, due to \eqref{E:Ind-i} one has $t= \lim_{i \to \infty} \frac{i}{\kappa_i}$ and the conclusion holds since \eqref{E:mai-c} and the fact $\kappa_i \to \infty$.
	
On the other hand, suppose that  $\underline{\sigma} < \overline{\sigma}$. For any $t \in [\underline{\sigma}, \ov{\sigma}]$, we shall construct an increasing sequence $i_l, l \in \BFN$ such that
\be
\lim_{l \to \infty} \frac{i_l-1}{\kappa_{i_l}} = t.\nonumber
\ee
Indeed, if such sequence exists, due to Theorem \ref{T:index}, there always exists a closed geodesic, denoted by $c$, is $((m_l-1)p+1,i_l)$-variationally visible. Then the desired conclusion follows from \eqref{E:mai-c} since $\frac{3(n-1)}{\kappa_{i_l}} \to 0$ as $l \to \infty$.

To construct this sequence, we shall follow an induction argument. Suppose that we have constructed $i_l > i_{l-1}$ satisfying
\be
|t - \frac{(i_l-1)}{\kappa_{i_l}}| \leq \frac{1}{l}.\nonumber
\ee
The definition of $\underline{\sigma}$ implies that one always find $j > i_l$ such that
\be
\frac{(j-1)}{\kappa_{j}}< \underline{\sigma} + \frac{1}{l+1}  \text{ and } \frac{3(n-1)}{\kappa_{j}} \leq \frac{1}{l+1}.\nonumber
\ee
We shall now construct $i_{l+1} >i_l$ with
\be
|t - \frac{i_{l+1}-1}{\kappa_{i_{l+1}}} | \leq \frac{1}{l+1}\nonumber
\ee
to finish the induction.

Since $\kappa_j$ is non-decreasing with $j$, one has
\be
\frac{j+k-1}{\kappa_{j+k}} - \frac{j+k-2}{\kappa_{j+k-2}} \leq \frac{1}{\kappa_{j+k}} \leq \frac{1}{l+1}\nonumber
\ee
On the other hand, due to the definition of $\ov{\sigma}$, there are some $k_j \in \BFN$ such that
\be
\frac{j+k_j-1}{\kappa_{j+k_j}} > \ov{\sigma} - \frac{1}{l+1}\nonumber
\ee
Therefore, one has the intervals
\be
I_{k}:=[ \frac{j+k-1}{\kappa_{j+k}}- \frac{1}{l+1}, \frac{j+k-1}{\kappa_{j+k}}+ \frac{1}{l+1}]\nonumber
\ee
for $0 \leq k \leq k_j$ covers $[\underline{\sigma},\ov{\sigma}]$. Then at least one of them, denoted by $I_{\ov{k}}$, contains $t$. We can set $l_{j+1} = j + \ov{k}$ and finish the proof.\end{proof}

Combining (\ref{E:Ind-i}) with (\ref{E:mai-c}) of Lemma \ref{L:mean-con}, note that $\kappa_i \to \infty$ as $i \to \infty$, we obtain:
\begin{corollary} \label{L:Dens}
Suppose there are finitely many distinct closed geodesics of class $[h]$ on $(S^n/\Gamma,F)$. Then for any
$c \in \mathcal{V}_\infty(S^n/\Gamma,F;[h])$, we have
	\be
	\alpha(c)=2\sigma,\nonumber
	\ee
where $\sigma=\underline{\sigma}=\ov{\sigma}$ under our assumptions.
\end{corollary}

\begin{proposition} \label{P:Generic}
	Let $M=(S^n/\Gamma, F)$ be the $n$-dimensional compact Finsler space form  with flag curvature $K$ satisfies $ \frac{\lambda^2}{(\lambda+1)^2} \frac{4k^2}{(n-1)^2} < K \leq 1$ for some $k=1,\ldots,n-2$ and let either $n$ be even or $k \geq \frac{n-1}{2}$, where $n\geq 3$.
If there are only finitely many distinct non-contractible closed geodesics of class $[h]$, then there are $k+1$ geometrically distinct closed geodesics $c_1,\ldots,c_{k+1}$ with the same mean average index.
\end{proposition}

This proof of this proposition follows from the same argument in the proof of Theorem 1.3 and Corollary \ref{L:Dens}.
\\

Finally, we are going to study generic properties of non-contractible closed geodesic. Denote by $\mathfrak{F}$ the set of $C^r$-Finsler metrics with $r\geq 4$ on compact manifold $M$. We define the following subset
\be \label{E:Generic-metric}
\CF : =\{ F \in \mathfrak{F}\;| c_1, c_2 \text{ are distinct closed geodesics } \Rightarrow \alpha_F(c_1) \neq  \alpha_F(c_2) \},
\ee
in which the subscript $F$ is addressed on the dependence of the mean average index on metrics.
Our aim is to show that $\CF$ is a residual set, i.e., countable intersections of open, dense sets, in $\mathfrak{F}$ in strong $C^4$ topology.

Given a Finsler metric $F:TM \to \BFR$ on a compact smooth manifold $M$,  one can defines a family of Riemannian metrics $g^y$ on
the tangent space $T_x M$ parameterized by unit vectors $y \in S_x M:=\{w\in T_x M ; F(w) = 1\}$. As the local chart
$U \subset TM$ is represented as  $(x_1,\ldots,x_n,y_1,\ldots,y_n)$, these metrics are given by
\be \label{E:Osc-R}
g(x,y) = g^y(x) =  \frac{1}{2} \frac{\p F^2}{\p y_i \p y_j}(x,y) dy_i dy_j, \;\; y \in T_x M.
\ee
For any closed geodesics $c$ parameterized
by arc length, we can extend the unit vector field $\dot{c}$ along $c$ to a non-zero vector field $V$ in the tubular neighborhood of $c$ and we actually obtain a Riemannian metric due to the positively definiteness  of $F^2$. Then the Riemannian metric $g^V$
is called {\it osculating Riemannian metric}. As we consider the normal bundle coordinates $(u_0,\ldots,u_{n-1}) = (u_0(t),u(t))$ in a tubular neighbourhood of the closed geodesic with respect to the osculating Riemannian metrics $g^V$. The closed geodesic is represented by $u_0(t)=t, u(t)=0$ and for any $t_0$ and $v\in \BFR^{n-1}$ the curve $u_0(t)=t_0, u(t)=tv$ is also a geodesic. Then for the metric $g^V$ we obtain
\be
g_{ij}^V(t,0)=\delta_{ij}, 0 \leq i,j \leq n-1, \; \p_{u_k} g_{ij}^V(t,0) = 0.\nonumber
\ee
If $c$ is a closed geodesic of $(M,F)$, it is also a closed geodesic on $(M,g^V)$ with the same length. Let $\mathcal{P} \in T_xM$ be a plane with $V(x) \in \mathcal{P}$, then the flag curvature $K(V(x);\mathcal{P})$ equals the sectional curvature $K(\mathcal{P})$ of $g^V$. Parallel transport
along $c$ as well as the index and nullity of the closed geodesic of the
Finsler metrics agree with the parallel transport and the index and
nullity of the osculating Riemannian metric. In a similar way given in section 3 of \cite{RT20} to extend the bumpy metric theorem to osculating Riemannian metrics, $\mathcal{F}$ is a residual set in strong $C^4$ topology due to a local perturbative argument from \cite{KT72}, which was used
by Rademacher to prove a similar conclusion with Riemannian metrics in $C^2$ topology, cf. \cite{Rad94}.

\begin{proof}[Proof of Theorem \ref{T:main-4}]
	Suppose that there are only finitely many distinct non-contractible closed geodesics of class $[h]$ on $(S^n/\Gamma,F)$ for $F\in \mathcal{F}$. Due to Proposition \ref{P:Generic},  we obtain there are at least two non-contractible closed geodesics in $\Lambda_h M$ if the flag curvature condition is satisfied. They share the same mean average index. On the other hand,  due to the definition of the Finsler metrics $\mathcal{F}$ given in \eqref{E:Generic-metric}, for $C^4$-generic metrics, there are only one closed geodesic, which is a contradiction.
\end{proof}

%\bibliography{cg}

%\bibliographystyle{plain}

\end{document}